\documentclass[12pt,twoside]{amsart}

\usepackage{amsmath, amssymb, amscd, paralist,tabularx,supertabular,amsmath, verbatim, amsthm, amssymb, mathrsfs, manfnt,times,latexsym,amscd,graphicx}
\usepackage{pstricks}
\usepackage[all]{xy}

\usepackage{fullpage}

\DeclareMathOperator{\Ram}{Ram}

\DeclareMathOperator{\PGL}{PGL}

\DeclareMathOperator{\isom}{Isom}

\DeclareMathOperator{\norm}{Norm}

\DeclareMathOperator{\Reg}{Reg}

\DeclareMathOperator{\Vol}{Vol}

\newcommand{\C}{\mathbb C}

\newcommand{\HH}{\mathbb H}

\newcommand{\Q}{\mathbb Q}

\newcommand{\frakp}{\mathfrak{p}}

\newcommand{\calD}{\mathcal{D}}

\newcommand{\calO}{\mathcal{O}}

\numberwithin{equation}{section}

\newtheorem{prop}{Proposition}[section]

\theoremstyle{remark}

\newtheorem{rmk}[prop]{Remark}

\newtheorem*{ack}{Acknowledgements}

\theoremstyle{plain}
\newtheorem{thm}[prop]{Theorem}

\newtheorem{cor}[prop]{Corollary}

\title{On fields of definition of arithmetic Kleinian reflection groups~II}

\author{Mikhail Belolipetsky}\thanks{Belolipetsky partially supported by a CNPq research grant.}
\address{IMPA\\
Estrada Dona Castorina 110\\
22460-320 Rio de Janeiro, Brazil}
\email[] {mbel@impa.br}

\author{Benjamin Linowitz}
\address{Department of Mathematics\\
530 Church Street\\
University of Michigan\\
Ann Arbor, MI 48109 USA}
\email[] {linowitz@umich.edu}

\dedicatory{Dedicated to the memory of Colin Maclachlan}

\begin{document}

\begin{abstract}

Following the previous work of Nikulin and Agol, Belolipetsky, Storm, and Whyte it is known that there exist only finitely many (totally real) number fields that can serve as fields of definition of arithmetic hyperbolic reflection groups. We prove a new bound on the degree $n_k$ of these fields in dimension $3$: $n_k$ does not exceed $9$. Combined with previous results of Maclachlan and Nikulin, this leads to a new bound $n_k \le 25$ which is valid for all dimensions. We also obtain upper bounds for the discriminants of these fields and give some heuristic results which may be useful for the classification of arithmetic hyperbolic reflection groups.

\end{abstract}

\maketitle

\section{Introduction}

A group of isometries of hyperbolic $n$-space $\HH^n$ is called a \emph{reflection group} if it has a finite generating set which consists of reflections in hyperplanes. The study of hyperbolic reflection groups has a long and remarkable history going back to the papers of Makarov and Vinberg. In recent years there has been a wave of activity in this area which has led to a solution to the open question of the finiteness of these groups and to some quantitative results towards their classification. An improvement of the quantitative bounds is the subject of the present paper.

Recall that a reflection group is called \emph{maximal} if it is not properly contained in any other reflection group. In many cases, while studying hyperbolic reflection groups it is useful to restrict one's attention to \emph{arithmetic groups} of isometries (see definition in Section~\ref{sec:prelim}). Answering a long-standing open question, it was proved independently in \cite{ABSW} and \cite{Nikulin07} that there are only finitely many conjugacy classes of arithmetic maximal hyperbolic reflection groups. This implies, at least theoretically, that these groups can be classified. Considering the ubiquity of hyperbolic reflection groups, their classification stands as a fundamental open problem. Our current knowledge falls short of a solution to this classification problem, though there has been considerable progress recently.

The proofs of the finiteness theorems in \cite{ABSW} and \cite{Nikulin07} are very different from one another. The method of \cite{Nikulin07} goes back to \cite{Nikulin81} and is based on a careful analysis of the local combinatorial structure of the Coxeter polyhedron of a reflection group, whereas the proof in \cite{ABSW} takes advantage of the global arithmetic and geometric properties of reflection groups. Both methods are effective in principle but the quantitative bounds which one can extract from the proofs are extremely large and have no practical value. Improvement of these bounds has been the subject of subsequent work. Currently, the best known general results are obtained through a combination of the two aforementioned methods (see however \cite{belolipetsky-congr} for a different approach in a restricted setting). This paper is not an exception; we will first obtain new quantitative bounds in dimension $n=3$ based on the method of Agol \cite{agol-reflection} (related to \cite{ABSW}), and then combine them with Nikulin's results from \cite{Nikulin07} and \cite{Nikulin11} to produce the best current bounds in higher dimensions.

Associated to an arithmetic group is a field $k$, which is called its field of definition. In our case, the field $k$ is a totally real algebraic number field. The degree $n_k$ and the discriminant $d_k$ of this field are the basic invariants of the arithmetic reflection groups we are interested in. Finiteness of arithmetic maximal hyperbolic reflection groups in dimension $3$ was first proved by Agol \cite{agol-reflection}, but his argument as it stands does not provide quantitative bounds for their invariants. This issue was addressed in \cite{belolipetsky-fields}, where explicit bounds for the degrees and discriminants of the fields of definition were obtained. In this paper, we use more careful number theoretic considerations which allow us to significantly improve the results in \cite{belolipetsky-fields}. Thus, in Theorem~\ref{thm:main}, we show that the degree $n_k \le 9$ (compare with $n_k \le 35$ in \cite{belolipetsky-fields}). We also give explicit upper bounds for the discriminant $d_k$ in each of the possible degrees (see Theorem~\ref{thm:discbounds}). These results bring the invariants of the defining fields close to the range of existing tables \cite{tables}. We also present a heuristic argument which allows us to further narrow the list (Theorem~\ref{thm:with2classnumberbound}). The proofs of the main theorems use Borel's volume formula \cite{borel-commensurability} and results of Chinburg and Friedman \cite{chinburg-smallestorbifold} (as in \cite{belolipetsky-fields}) together with a Laplace-eigenvalue bound of Luo, Rudnick and Sarnak \cite{LRS}, a refinement of the Odlyzko discriminant bounds that takes into account primes of small norm \cite{Poitou}, \cite{doud} and Louboutin's improvement of the Brauer--Siegel theorem \cite{Louboutin}. We also note that several aspects of our proof were inspired by the work of Doyle, Linowitz and Voight \cite{DLV} on isospectral but not isometric arithmetic $2$- and $3$-manifolds of small volume. The details are given in Sections \ref{sec:main} and \ref{sec:discr}. In the last section, we show how these results may be combined with the previous work of Maclachlan and Nikulin in order to deduce a general upper bound $n_k \le 25$ which is valid for all dimensions.

\begin{ack} We would like to thank John Voight for helpful conversations.\end{ack}

\section{Preliminaries}\label{sec:prelim}

The group of orientation preserving isometries of the hyperbolic $3$-space $\isom^+(\HH^3)$ is isomorphic to $\PGL(2,\C)$. In this section, we will recall the definition of arithmetic subgroups of $\PGL(2, \C)$. We refer to \cite{borel-commensurability} and \cite{MR} for further details and examples of arithmetic subgroups.

Let $K$ be a number field with exactly one complex place, $\calO_K$ its ring of integers, and $B$ a quaternion algebra over $K$. Let $\calD$ be a maximal $\calO_K$-order of $B$, denote by $\calD^1$ its group of elements of norm $1$ and let $\norm(\calD^1)$ be the normalizer of $\calD^1$ in $B$. Consider a $K$-embedding $\rho: B \hookrightarrow \mathrm{M}(2,\C)$ associated with the complex place of $K$. The projection
$$\Gamma_\calD = P\rho(\norm(\calD^1)) < \PGL(2,\C),$$
where $P: \mathrm{M}(2,\C) \to \PGL(2,\C)$, is then a discrete finite covolume subgroup of $\PGL(2,\C)$. Any subgroup of $\PGL(2,\C)$ which is commensurable with some such group $\Gamma_\calD$ is called an \emph{arithmetic subgroup} and the field $K$ is called its \emph{field of definition}.

A discrete subgroup of $\PGL(2,\C)$ is called \emph{maximal} if it is maximal within its commensurability class with respect to inclusion. It can be shown that the groups $\Gamma_\calD$ described above are in fact maximal arithmetic subgroups of $\PGL(2,\C)$. These are not the only maximal arithmetic subgroups; by Borel's theorem the commensurability class of an arithmetic subgroup contains infinitely many maximal elements which can be described explicitly \cite{borel-commensurability}. For the purpose of this article, however, the representatives $\Gamma_\calD$ will suffice.

In order to deal with hyperbolic reflection groups, we need to consider more general \emph{Kleinian groups}, which are the discrete subgroups of the full group of isometries $\isom(\HH^3)$. All the above cited material can be applied to a reflection group $\Gamma$ by considering its index $2$ orientation preserving subgroup $\Gamma^+$, for which we have $\Gamma^+ < \PGL(2,\C)$. With regards to arithmeticity, it should be pointed out that arithmetic reflection groups fall into a special class of \emph{arithmetic subgroups defined by quadratic forms} \cite[Lemma 7]{Vinberg67}. It follows that in this case the defining field $K$ is of \emph{even degree} and has a totally real subfield $k$ such that $[K:k] = 2$ (cf. \cite[Theorem~10.4.1]{MR}). The latter are the fields of definition which were discussed in the introduction. In order to stress the difference between $K$ and $k$, we will always refer to the latter as the \emph{totally real field of definition} of an arithmetic (reflection) group.

\section{Main theorem}\label{sec:main}

In this section, we shall prove our main result.

\begin{thm}\label{thm:main}
Let $K$ be the field of definition of an arithmetic Kleinian reflection group. Then the degree of $K$ is at most $18$.
\end{thm}

\begin{cor}
The degree of the totally real field of definition of an arithmetic hyperbolic reflection group in dimension $3$ is at most $9$.
\end{cor}

Borel's volume formula shows that the group $\Gamma_\calD$ has minimal volume within its commensurability class \cite[Theorem~5.3]{borel-commensurability}. It therefore suffices to exhibit an upper bound $V$ for the covolume of a maximal arithmetic Kleinian group which may contain a reflection group and show that if $\Gamma_\calD$ is a maximal arithmetic Kleinian group which has covolume less than $V$ and whose field of definition $K$ is of even degree, then $K$ has degree at most $18$. Agol \cite{agol-reflection} has shown that one may take $V=128\pi^2$. In his derivation of this bound, Agol makes critical use of the Burger--Sarnak--Vign\'eras lower bound of $3/4$ for the minimal nonzero eigenvalue of the Laplacian on $\mathbb H^3/\Gamma$. This bound has subsequently been improved by a number of authors, allowing us to obtain a slightly stronger volume bound.

\begin{prop}\label{prop:volumebound} The volume of a maximal arithmetic Kleinian group which contains a reflection group is at most $108\pi^2$.
\end{prop}
\begin{proof}
Let $\Gamma$ be a maximal arithmetic Kleinian group containing a reflection subgroup. In his proof \cite[Theorem 6.1]{agol-reflection} that the covolume of $\Gamma$ is at most $128\pi^2$, Agol showed that 
$$\lambda_1(\mathbb H^3/\Gamma)\left(\frac12\cdot\Vol(\mathbb H^3/\Gamma)\right)^{2/3}\leq 3(8\pi^2)^{2/3},$$ 
where $\lambda_1(\mathbb H^3/\Gamma)$ is the minimal nonzero eigenvalue of the Laplacian on $\mathbb H^3/\Gamma$.

Note that a maximal arithmetic Kleinian group is congruence \cite[Lemma 5.1]{agol-reflection} (see also \cite[Lemma 4.2]{LMR}). The proposition now follows from Luo, Rudnick and Sarnak's \cite{LRS} proof that $\lambda_1(\mathbb H^3/\Gamma)\geq 21/25$ whenever $\Gamma$ is a congruence subgroup of an orthogonal group defined by a quadratic form over a totally real number field.
\end{proof}

\begin{rmk}
We remark that the generalized Ramanujan conjecture implies that the minimal nonzero eigenvalue of the Laplacian of a congruence arithmetic Kleinian group is $\ge 1$. This would lead to a volume bound of $84\pi^2$.
\end{rmk}

Fix a maximal arithmetic Kleinian group $\Gamma_\calD$ with defining field $K$ of even degree $n$ and defining quaternion algebra $B$. It is well known that such a field $K$ must contain a unique complex place. Denote by $\Ram_f(B)$ the set of finite primes of $K$ which ramify in $B$. By Borel's volume formula \cite{borel-commensurability}, we have (cf. \cite[Proposition~2.1]{chinburg-smallestorbifold}):

\begin{equation}\label{equation:volumeinequality}
\Vol(\mathbb H^3/\Gamma_\calD)\geq \frac{8 \pi^2\zeta_K(2)d_K^{3/2}[\calO_K^\times:\calO_{K,+}^\times]\prod_{\frakp\in\Ram_f(B)}\left(\frac{N(\frakp)-1}{2}\right)}{(8\pi^2)^{n_K} h(K,2,B)}.
\end{equation}

Here, $\zeta_K$ denotes the Dedekind zeta function of $K$, $d_K$ the absolute value of the discriminant of $K$, $\calO_K^\times$ the units of $K$, $\calO_{K,+}^\times$ the totally positive units of $K$, $n_K$ the degree of $K$,  and $h(K,2,B)$ the order of the (wide) ideal class group of $K$ modulo the squares of all classes and the classes corresponding to primes in $\Ram_f(B)$.

Denote by $m$ the rank (over $\mathbb F_2$) of the group of totally positive units of $K$ modulo squares so that $[\calO_K^\times:\calO_{K,+}^\times]=2^{n_K-1-m}$. Let $\omega_2(B)$ denote the number of primes of $K$ that have norm $2$ and which ramify in $B$. Recall that $\zeta_K(s)$ has an Euler product expansion $\zeta_K(s)=\prod_\frakp \left(1-N(\frakp)^{-s}\right)^{-1}$ that converges for $\mathfrak{Re}(s)>1$. It follows that $\zeta_K(2)\geq (4/3)^{\omega_2(B)}$. If $\frakp$ is a prime of $K$ which ramifies in $B$, then $\frac{N(\frakp)-1}{2}\geq 1$ unless $N(\frakp)=2$. It follows that $\prod_{\frakp\in\Ram_f(B)}\left(\frac{N(\frakp)-1}{2}\right)\geq (1/2)^{\omega_2(B)}$.

Our proof of Theorem \ref{thm:main} will follow from a careful analysis of the number theoretic quantities present in (\ref{equation:volumeinequality}). All of our calculations were performed with the open-source computer algebra system Sage \cite{sage}.

We begin by combining Proposition \ref{prop:volumebound} with (\ref{equation:volumeinequality}) to deduce

\begin{equation}\label{equation:volumebound}
108\pi^2 \geq \Vol(\mathbb H^3/\Gamma_\calD)\geq \frac{8 \pi^2 (2/3)^{\omega_2(B)}2^{n_K-1-m} d_K^{3/2}}{(8\pi^2)^{n_K}h(K,2,B)}.
\end{equation}

We wish to derive an upper bound for the root discriminant $\delta_K:=d_K^{1/n_K}$ of $K$. To do so, we must first bound the value of $h(K,2,B)$. Lemma 4.3 of \cite{chinburg-smallestorbifold} implies

\begin{equation}\label{equation:volumeinequality2}
108\pi^2 \geq \Vol(\mathbb H^3/\Gamma_\calD) > 0.69 \exp\left(0.37n_K - \frac{19.08}{h(K,2,B)}\right).
\end{equation}

If $n_K\geq 20$, then equation (\ref{equation:volumeinequality2}) implies that $h(K,2,B)\leq 333$. Because $h(K,2,B)$ is a power of $2$, we see that $h(K,2,B)\leq 256$. By employing the trivial bounds $\omega_2(B)\leq n_K$ and $m\leq n_K-1$ (the latter follows from Dirichlet's unit theorem) we see from (\ref{equation:volumebound}) that

\begin{equation}\label{equation:odlyzko}
\delta_K \leq 24.117\cdot 229^{1/n_K}.
\end{equation}

If $n_K\geq 38$, then the discriminant bounds of Odlyzko \cite{Odlyzko-bounds} (see also Martinet \cite{Martinet}) imply that $\delta_K\geq 28.730$. This contradicts (\ref{equation:odlyzko}), hence $n_K\leq 36$.

We now show that $n_K\neq 20$. The proof that $n_K\not\in\{ 22,24,26,28,30, 32, 34, 36\}$ is similar.

\begin{prop}\label{prop:nofield} No maximal arithmetic Kleinian group with volume $< 108\pi^2$ has a defining field $K$ of degree $20$.
\end{prop}
\begin{proof}
Suppose that a maximal arithmetic Kleinian group with volume $< 108\pi^2$ has field of definition $K$ of degree $20$. Employing the bound $h(K,2,B)\leq 256$ along with the trivial bounds $\omega_2(B)\leq n_K$ and $m\leq n_K-1$ in (\ref{equation:volumebound}) shows that $\delta_K\leq 31.646$. We now employ a refinement of the Odlyzko bounds that takes into account the existence of primes of small norm, due to Poitou \cite{Poitou} and further developed by Brueggeman and Doud \cite{doud}. The Odlyzko-Poitou bounds show that the root discriminant of a number field of degree $20$ with $18$ real places and at least $6$ primes of norm $2$ is at least $33.387$. This is a contradiction and allows us to deduce that $\omega_2(B)\leq 5$. We now return to (\ref{equation:volumebound}) to obtain a smaller upper bound for $\delta_K$ and then again utilize the Odlyzko--Poitou discriminant bounds in order to obtain a better bound for $\omega_2(B)$. Repeating this process shows that $\omega_2(B)=2$ and $\delta_K\leq 24.810$.

We claim that the class number $h_K$ of $K$ is $1$. If $h_K\geq 2$, then the Hilbert class field of $K$ is a number field of degree $20\cdot h_K\geq 40$ with $18\cdot h_K\geq 36$ real places and with root discriminant $\delta_K$. Applying the Odlyzko bounds to this field shows that $\delta_K\geq 27.950$. This is a contradiction and proves our claim.

We now employ a theorem of Armitage and Fr\"ohlich \cite{armitage-frolich} in order to deduce that $\lfloor n_K/2 \rfloor\geq m$. Equation (\ref{equation:volumebound}) now implies that $\delta_K < 16.751$ for all $\omega_2(B)\leq 2$. The Odlyzko bounds imply that no field of degree $20$ with $18$ real places has root discriminant $<19.365$. This contradiction finishes our proof.
\end{proof}

\section{Discriminant bounds and class numbers}\label{sec:discr}

In this section, we deduce explicit upper bounds for the root discriminant of the defining field $K$ of an arithmetic Kleinian reflection group. The case in which the degree $n_K$ of $K$ is equal to $2$ was handled completely in \cite{belolipetsky-fields}, hence we assume $4\leq n_K\leq 18$.

In order to obtain reasonable discriminant bounds, we must first obtain an upper bound for $h(K,2,B)$. We note that (\ref{equation:volumeinequality2}) no longer provides us with such a bound, as the inequality always holds for $n_K\leq 18$. Instead, we will make use of Louboutin's refinement of the Brauer--Siegel theorem \cite{Louboutin}.

Let $\kappa_K$ denote the residue at $s=1$ of the Dedekind zeta function $\zeta_K(s)$ of $K$. By the analytic class number formula \cite[Chapter XIII]{lang} we have:
\begin{equation}\label{equation:analyticclassnumberformula}
\kappa_K=\frac{(2\pi)h_K\Reg_K 2^{n_K-2}}{w_Kd_K^{1/2}},
\end{equation}
where $h_K$ is the class number of $K$, $\Reg_K$ the regulator of $K$ and $w_K$ the number of roots of unity contained in $K$. Note that our assumption that $n_K>2$ implies that there exists an embedding of $K$ into $\mathbb R$, hence $w_K=2$.

As $h(K,2,B)\leq h_K$, we may derive an upper bound for $h(K,2,B)$ by combining an upper bound for $\kappa_K$ with a lower bound for $\Reg_K$.

Louboutin (see \cite[Prop. 2]{Louboutin} and \cite[Thm. 1]{Louboutin-upperbounds}) has shown  that
\begin{equation}\label{equation:louboutin}
	\kappa_K\leq \left(\frac{e\log(d_K)}{2(n_K-1)}\right)^{n_K-1}
\end{equation}
holds for all $K$ and that when $K$ is a quartic field with root discriminant $\delta_K\geq 17$ one has
\begin{equation}\label{equation:louboutinquartic}
	\kappa_K\leq \frac{\log^{n_K-1}(d_K)}{2^{n_K-1}\left(n_K-1\right)!}.
\end{equation}

We note that in order to apply \cite[Thm. 1]{Louboutin-upperbounds}, which implies (\ref{equation:louboutinquartic}), it is necessary that the quotient of zeta functions $\zeta_K(s)/\zeta(s)$ be entire. This is well known when $n_K=4$.

Following the work of Zimmert \cite{Zimmert}, Friedman \cite[pages 620-621]{Friedman} has shown that
\begin{equation}\label{equation:friedmanbound}
	\Reg_K \geq 0.0062 e^{(0.241n_K+0.497(n_K-2))},
\end{equation}
and that $\Reg_K\geq 0.36$ when $n_K=4$ and $\Reg_K\geq 1.23$ when $n_K=6$. Combining (\ref{equation:analyticclassnumberformula}), (\ref{equation:louboutin}) and (\ref{equation:friedmanbound}) yields:
\begin{equation}\label{equation:hbound}
h(K,2,B) \leq \frac{d_K^{1/2}\left(e\log(d_K)\right)^{n_K-1}}{0.0062\pi 2^{2n_K-3} \left(n_K-1\right)^{n_K-1}  e^{(0.241n_K+0.497(n_K-2))}}.
\end{equation}

When $n_K=4$, we get the improved bound 
\begin{equation}\label{equation:hboundquartic}
h(K,2,B) \leq \frac{d_K^{1/2}\log^{n_K-1}(d_K)}{0.36\pi 2^{2n_K-3
} \left(n_K-1\right)!},
\end{equation}
and when $n_K=6$, we get
\begin{equation}\label{equation:hboundsextic}
h(K,2,B) \leq \frac{d_K^{1/2}\left(e\log(d_K)\right)^{n_K-1}}{1.23\pi 2^{2n_K-3} \left(n_K-1\right)^{n_K-1}}.
\end{equation}

It is clear that there are only finitely many values of $d_K$ for which both (\ref{equation:volumeinequality}) and (\ref{equation:hbound}) (or (\ref{equation:hboundquartic}), (\ref{equation:hboundsextic})  if $n_K=4, 6$) can hold. In the table below we list, for every degree $4\leq n_K \leq 18$, the associated upper bound for $\delta_K$.

\begin{equation}\label{equation:disctable}
\begin{array}{c|c}
n_K & \delta_K  \\
\hline
\rule{0pt}{2.5ex}
4 & < 668 \\
6 & 649 \\
8 & 639  \\
10 & 503 \\
12 & 445 \\
14 & 395 \\
16 & 361 \\
18 & 346 \\

\end{array}
\end{equation}

As the root discriminant of any subfield of $K$ does not exceed $\delta_K$, we may summarize the above discussion as follows.

\begin{thm}\label{thm:discbounds}
The root discriminant of the totally real field of definition $k$ of an arithmetic hyperbolic reflection group in dimension $3$ is bounded above by the entry corresponding to $2n_k$ in (\ref{equation:disctable}).
\end{thm}

Using the computer algebra system Sage, one may easily enumerate totally real fields of small degree and discriminant. The following corollaries are immediate consequences of Theorem \ref{thm:discbounds}.

\begin{cor} There are at most $135,643$ real quadratic fields which are the totally real field of definition of an arithmetic hyperbolic reflection group in dimension $3$.
\end{cor}

\begin{cor} There are at most $17,449,721$ real cubic fields which are the totally real field of definition of an arithmetic hyperbolic reflection group in dimension $3$.
\end{cor}

While the root discriminant bounds in (\ref{equation:disctable}) represent a major step towards the enumeration of all totally real defining fields of arithmetic Kleinian reflection groups (compared for instance, with the discriminant bounds produced in \cite{belolipetsky-fields}), they are still in general much too large for the fields to appear in existing tables of totally real number fields \cite{tables}. The primary reason that the discriminant bounds in (\ref{equation:disctable}) are as large as they are is that because we lack a good bound for $h(K,2,B)$ we are forced to bound $h(K,2,B)$ by the full class number of $K$. We then bound the class number of $K$ using Louboutin's refinement of the Brauer-Siegel theorem \cite{Louboutin}. It seems likely that these bounds are far from optimal and that every field satisfying our root discriminant bounds has 2-class number considerably smaller than the bounds produced by refinements of the Brauer--Siegel theorem.

Recall that if $G$ is a finite abelian group and $p$ is a prime number, then the \emph{$p$-rank} of $G$ is the number of factors of $p$-power order which occur in the direct sum decomposition of $G$ into cyclic primary components. We conclude this section by classifying all totally real defining fields of arithmetic Kleinian reflection groups whose defining fields $K$ have class group with $2$-rank at most $12$. These fields have associated $h(K,2,B)$ values sufficiently small that the methods used in the proof of Proposition \ref{prop:nofield} can be employed in order to show the degree of such a field is at most $6$ and produce root discriminant bounds which allow for their enumeration. Specifically, we get the following theorem.

\begin{thm}\label{thm:with2classnumberbound}
Let $K$ be the defining field of an arithmetic Kleinian reflection group $\Gamma$ and assume that the $2$-rank of the class group of $K$ is at most $12$. Then the totally real defining field of $\Gamma$ has degree at most $6$ and appears in Voight's table of totally real number fields \cite{tables}.
\end{thm}

We also include a table of the root discriminant bounds produced for each even degree $4\leq n_K\leq 12$ and remind the reader that for each degree $n_K$, any totally real subfield of $K$ will also have root discriminant $\le \delta_K$.

\begin{equation}\label{equation:disctable2}
\begin{array}{c|c}
n_K & \delta_K  \\
\hline
\rule{0pt}{2.5ex}
4 & 149 \\
6 & 82 \\
8 & 60 \\
10 & 29  \\
12 & 16 \\
\end{array}
\end{equation}

\begin{rmk}
We briefly state an interesting consequence of the failure of the hypothesis of Theorem \ref{thm:with2classnumberbound}. Consider for instance, the defining field $K$ of an arithmetic hyperbolic Kleinian reflection group with class group having $2$-rank $\ge 12$. By Theorem \ref{thm:main}, the degree $n_K$ of $K$ is at most $18$. It now follows from the theorem of Golod--Shafarevich \cite{golod-shafarevich} (see also \cite[Theorem 3]{roquette}) that $K$ has an infinite $2$-class field tower. Odlyzko \cite{Odlyzko-bounds} has shown that in this situation the root discriminant $\delta_K$ of $K$ satisfies $\delta_K \geq (60.8395)^{(n_K-2)/n_K}\cdot (22.3816)^{(2/n_K)}$ unconditionally, and that $\delta_K \geq (215.3325)^{(n_K-2)/n_K}\cdot (44.7632)^{(2/n_K)}$ if one assumes the generalized Riemann hypothesis. These bounds can be combined with those of (\ref{equation:disctable}) in order to obtain a narrower interval in which $\delta_K$ can lie. It should be noted that the existence of number fields having a unique complex place, infinite $2$-class field tower and root discriminants satisfying the upper bounds of (\ref{equation:disctable}) would be extremely significant. For instance, the smallest known (in terms of root discriminant) totally real field with infinite class field tower has root discriminant $1058.56...$ \cite{martinet2}.
\end{rmk}

\section{Degree bounds for higher dimensions}\label{sec:dim}

In this section, we will discuss some higher dimensional implications of our main theorem.

Let $\Gamma$ be an arithmetic hyperbolic reflection group in dimension $n$. For $n \neq 3$, the definition of arithmeticity given in Section~\ref{sec:prelim} would not apply. Here we are going to use another approach which is more specific to reflection groups. We first introduce some necessary terms.

Consider a fundamental polyhedron $P \subset \HH^n$ for a discrete group $\Gamma < \isom(\HH^n)$ generated by reflections in the sides of $P$. Let $G(P) = [a_{ij}]_{i,j = 1}^m$ denote the Gram matrix of $P$ (cf. e.g. \cite[Section~10.4]{MR}). Define two fields $\tilde{k}(P)$ and $k(P)$ as follows:
\begin{align*}
\tilde{k}(P) &= \Q(\{a_{ij}:\ i,j = 1,2,\ldots,m\}),\\
k(P)         &= \Q(\{b_{i_1i_2\ldots i_r}\}),
\end{align*}
where for all subsets $\{i_1,i_2,\ldots,i_r\} \subset \{1,2,\ldots, m\}$, we define the cyclic products $b_{i_1i_2\ldots i_r}$ by 
$$b_{i_1i_2\ldots i_r} = a_{i_1i_2}a_{i_2i_3}\ldots a_{i_ri_1}.$$

Vinberg~\cite{Vinberg67} proved that the group $\Gamma$ is \emph{arithmetic} if $P$ has finite volume, all the cyclic products $b_{i_1i_2\ldots i_r}$ are algebraic integers, the field $\tilde{k}(P)$ is totally real, and $^\sigma G(P) = [\sigma(a_{ij})]$ is positive semi-definite for all Galois embeddings $\sigma : \tilde{k}(P)\hookrightarrow \C$ such that $\sigma|k(P) = \mathrm{Id}$. It can be checked that when $n = 3$ the field $k(P)$ coincides with the totally real field of definition of $\Gamma$ considered above.

Based on Vinberg's criterion, Nikulin \cite{Nikulin81} introduced the notion of $V$-arithmeticity: a reflection group $\Gamma$ is called \emph{$V$-arithmetic} if $P$ satisfies all the conditions of Vinberg's criterion except possibly the requirement that it has finite volume. This notion allows us to apply induction on the sides of the arithmetic hyperbolic polyhedra and to reduce certain questions to smaller dimension. It plays a central role in Nikulin's work on arithmetic reflection groups.

Another important ingredient of Nikulin's method is the notion of minimality. Given a positive real number $t$, the polyhedron $P$ is called \emph{$t$-minimal} if $a_{ij} < t$ for all entries $a_{ij}$ of the Gram matrix $G(P)$. 

The main result of \cite{Nikulin07} states that the degree of the totally real field of definition of an arithmetic hyperbolic reflection group does not exceed the maximum of the degrees of the possible fields in dimensions $n = 2, 3$ and the \emph{transition constant} $N(14)$. The constant $N(14)$ is equal to the maximal degree of the totally real fields of definition of certain $V$-arithmetic groups with $4$ generators whose fundamental polyhedra have minimality $14$ (see \cite{Nikulin81, Nikulin11}). It was proved in \cite{Nikulin11} that $N(14)$ is bounded above by $25$.

Our main theorem implies that $n_k \le 9$ for dimension $n = 3$. Maclachlan proved in \cite{maclachlan} that $n_k \le 11$ for $n = 2$. Together with the above cited results of Nikulin these bounds imply the following corollary.

\begin{cor}
The degree of the totally real fields of definition of arithmetic hyperbolic reflection groups in all dimensions is at most $25$.
\end{cor}

Let us mention that before our work similar methods were used to show that the degree is bounded above by $35$, and the critical place was in dimension $3$. Our main result allows us to shift the attention back to the value of the transition constant, where further improvements can be expected.


\begin{thebibliography}{11}
	
	\bibitem{agol-reflection}
	I.~Agol,
	\newblock Finiteness of arithmetic {K}leinian reflection groups,
	\newblock In {\em International {C}ongress of {M}athematicians. {V}ol. {II}},
	  pages 951--960. Eur. Math. Soc., Z\"urich, 2006.

\bibitem{ABSW} I.~Agol, M.~Belolipetsky, P.~Storm, K.~Whyte,
Finiteness of arithmetic hyperbolic reflection groups,
{\em Groups, Geometry, and Dynamics}, {\bf 2} (2008), 481--498.

	\bibitem{armitage-frolich}
	J.~V. Armitage and A.~Fr{\"o}hlich,
	\newblock \emph{Class numbers and unit signatures},
	\newblock Mathematika, \textbf{14} (1967), 94--98.
	
	\bibitem{belolipetsky-fields}
	M.~Belolipetsky,
	\newblock On fields of definition of arithmetic {K}leinian reflection groups,
	\newblock {\em Proc. Amer. Math. Soc.}, {\bf 137} (2009), 1035--1038.

\bibitem{belolipetsky-congr}  M.~Belolipetsky, Finiteness theorems for congruence reflection groups, {\em Transform. Groups},
{\bf 16} (2011), 939--954.
	
	\bibitem{borel-commensurability}
	A.~Borel,
	\newblock Commensurability classes and volumes of hyperbolic {$3$}-manifolds,
	\newblock {\em Ann. Scuola Norm. Sup. Pisa Cl. Sci. (4)}, {\bf 8} (1981), 1--33.

	\bibitem{doud}
	S.~Brueggeman and D.~Doud,
	\newblock Local corrections of discriminant bounds and small degree extensions
	  of quadratic base fields,
	\newblock {\em Int. J. Number Theory}, {\bf 4} (2008), 349--361.

	\bibitem{chinburg-smallestorbifold}
	T.~Chinburg and E.~Friedman,
	\newblock The smallest arithmetic hyperbolic three-orbifold,
	\newblock {\em Invent. Math.}, {\bf 86} (1986), 507--527.


	\bibitem{DLV}
	P.~Doyle, B.~Linowitz and J.~Voight,
	\newblock The smallest isospectral and nonisometric orbifolds of dimension $2$ and $3$, preprint.
	
\bibitem{Friedman}
E.~Friedman,
\newblock Analytic formulas for the regulator of a number field,
\newblock {\em Invent. Math.}, {\bf 98} (1989), 599--622.

\bibitem{golod-shafarevich} E.~Golod and I.~Shafarevich, On the class field tower (Russian), \emph{Izv. Akad. Nauk SSSR Ser. Mat.}, {\bf 28} (1964), 261--272.

\bibitem{lang}
S.~Lang,
\newblock {\em Algebraic number theory}, Grad. Texts in Math. {\bf 110},
\newblock Springer-Verlag, New York, second edition, 1994.

	\bibitem{LMR}
	D.~Long, C.~Maclachlan, A.~Reid, Arithmetic Fuchsian groups of genus zero, \emph{Pure Appl.~Math.~Q.}, \textbf{2} (2006), no.~2, 569--599.

	\bibitem{Louboutin} S.~Louboutin, The {B}rauer-{S}iegel theorem, \emph{J. London Math. Soc. (2)}, {\bf 72} (2005), 40--52.
	
	\bibitem{Louboutin-upperbounds} S.~Louboutin, Upper bounds for residues of Dedekind zeta functions and class numbers of cubic and quartic number fields, \emph{Math. Comp.}, {\bf 80} (2011), no. 275, 1813--1822.	
	
	\bibitem{LRS}
	W.~Luo, Z.~Rudnick, P.~Sarnak, On the generalized Ramanujan conjecture for GL(n), \emph{Automorphic forms, automorphic representations, and arithmetic
	              ({F}ort {W}orth, {TX}, 1996)}, 301--310, \emph{Proc. Symposia in Pure Math.}, {\bf 66}, Part 2, Amer. Math. Soc., Providence, RI, 1999.
	
\bibitem{maclachlan} C.~Maclachlan, Bounds for discrete hyperbolic arithmetic reflection groups in dimension 2,
{\em Bull. Lond. Math. Soc.}, {\bf 43} (2011), 111--123.

\bibitem{MR}
C.~Maclachlan and A.~W.~Reid, {\em The Arithmetic of Hyperbolic 3-Manifolds},
Grad. Texts in Math. {\bf 219}, Springer (2003).


\bibitem{martinet2}
J. Martinet,
\newblock Tours de corps de classes et estimations de discriminants,
\newblock {\em Invent. Math.}, {\bf 44} (1978), 65--73.

	\bibitem{Martinet}
	J.~Martinet,
	\newblock \emph{Petits discriminants des corps de nombres},
	\newblock Number theory days, 1980 ({E}xeter, 1980), London Math.\ Soc.\ Lecture Note Ser., vol.\ 56, Cambridge Univ.\ Press,
	Cambridge, 1982, 151--193.


\bibitem{Nikulin81} V.~V.~Nikulin, On the classification of arithmetic groups generated by reflections in Lobachevsky spaces,
{\em Izv. Akad. Nauk SSSR Ser. Mat.}, {\bf 45} (1981), 113--142.
English transl. in {\em Math. USSR-Izv.}, {\bf 18} (1982), 99--123.

\bibitem{Nikulin07} V.~V.~Nikulin, Finiteness of the number of arithmetic groups generated by reflections
in Lobachevsky spaces, {\em Izvestiya: Mathematics}, {\bf 71} (2007), 53--56.

\bibitem{Nikulin11} V.~V.~Nikulin, The transition constant for arithmetic hyperbolic reflection groups,
{\em Izvestiya: Mathematics}, {\bf 75} (2011), 971--1005.
	
	\bibitem{Odlyzko-bounds}
	A.~M. Odlyzko,
	\newblock \emph{Bounds for discriminants and related estimates for class numbers,
	regulators and zeros of zeta functions: a survey of recent results},
	\newblock S\'em.\ Th\'eor.\ Nombres Bordeaux (2), \textbf{2} (1990), no.\ 1, 119--141.
	
	\bibitem{Poitou}
	G.~Poitou, \emph{Sur les petits discriminants}, S\'eminaire Delange-Pisot-Poitou, 18e
	ann\'ee: (1976/77), Th\'eorie des nombres, Fasc.\ 1 (French), Secr\'etariat Math., Paris, 1977, Exp. No. 6, 18 pp.
			
			
\bibitem{roquette} P. Roquette, {\em On class field towers.} In: J. W. S. Cassels and A. Frolich, Editors, Algebraic Number Theory, Academic Press, New York/London (1967), pp. 231--249.
				
	\bibitem{sage}
	W.~Stein et~al,
	\newblock {\em {S}age {M}athematics {S}oftware ({V}ersion 4.8)},
	\newblock The Sage Development Team, 2012.
	\newblock {\tt http://www.sagemath.org}.


\bibitem{Vinberg67} \`E.~B.~Vinberg, Discrete groups generated by reflections in Lobachevskii spaces,
{\em Mat. Sb. (N.S.),} {\bf 72} (114) (1967), 471--488;
correction, {\em ibid.,} {\bf 73} (115) (1967), 303.
English transl. in {\em Math. USSR-Sb.,} {\bf 1} (1967), 429--444.

\bibitem{tables} J.~Voight, Tables of Totally Real Number Fields. {\tt http://www.cems.uvm.edu/\~{}jvoight/nf-tables}.

\bibitem{Zimmert}
R.~Zimmert,
\newblock Ideale kleiner {N}orm in {I}dealklassen und eine
  {R}egulatorabsch\"atzung,
\newblock {\em Invent. Math.}, {\bf 62} (1981), 367--380.

\end{thebibliography}
\end{document}